\documentclass{article}
\usepackage[utf8]{inputenc}

\usepackage{amssymb}
\usepackage{amsmath}
\usepackage{amsthm}
\usepackage[pdftex]{graphicx}
\usepackage{tikz}
\usepackage{url}
\usepackage{color}
\usepackage{xifthen}
\usepackage{xcolor}
\usepackage{mathtools}
\usepackage{tablefootnote}
\usepackage[backend=biber,style=alphabetic,sorting=nty,doi=false,isbn=false,url=false,eprint=true]{biblatex}
\usepackage{hyperref}
\usepackage{geometry}
\usepackage[affil-it]{authblk}
\usepackage{here}

\usetikzlibrary{positioning}
\usetikzlibrary{calc}
\usetikzlibrary{decorations.pathreplacing}
\usetikzlibrary{cd}

\newcommand{\Z}{\mathbb{Z}}

\newcommand{\Q}{\mathbb{Q}}
\newcommand{\R}{\mathbb{R}}

\newcommand{\frakb}{\mathfrak{b}}
\newcommand{\frakd}{\mathfrak{d}}
\newcommand{\frakc}{\mathfrak{c}}

\newcommand{\frakr}{\mathfrak{r}}
\newcommand{\Pow}{\mathcal{P}}
\newcommand{\non}{\operatorname{non}}

\newcommand{\add}{\operatorname{add}}
\newcommand{\cof}{\operatorname{cof}}

\newcommand{\nul}{\mathsf{null}}
\newcommand{\meager}{\mathsf{meager}}

\newcommand{\AND}{\mathbin{\&}}

\newcommand{\restrict}{\upharpoonright}

\newcommand{\cp}{\mathfrak{cp}}
\newcommand{\icp}{\mathfrak{icp}}
\newcommand{\zero}{\mymathbb{0}}

\newcommand{\WD}{\mathbf{wd}}
\newcommand{\TD}{\mathcal{D}}
\newcommand{\leT}{\le_\mathrm{T}}
\newcommand{\ZFC}{\mathsf{ZFC}}

\newcommand{\ICP}{\mathbf{ICP}}
\newcommand{\frakB}{\mathbf{B}}

\newcommand{\BIP}{\mathbf{B}^\mathrm{IP}}
\newcommand{\IP}{\mathrm{IP}}
\newcommand{\Match}{\mathrm{Match}}
\newcommand{\Borel}{\mathsf{Borel}}
\newcommand{\leRK}{\le_{\mathrm{RK}}}
\newcommand{\geRK}{\ge_{\mathrm{RK}}}
\newcommand{\fin}{\mathsf{fin}}
\newcommand{\I}{\mathcal{I}}
\newcommand{\mmc}{\mathfrak{mc}}
\newcommand{\mmac}{\mathfrak{mac}}
\newcommand{\comp}{\mathrm{c}}

\newcommand{\seq}[1]{{\langle#1\rangle}}
\DeclarePairedDelimiter\abs{\lvert}{\rvert}

\renewcommand\emptyset{\varnothing}
\renewcommand\subset{\subseteq}
\renewcommand{\setminus}{\smallsetminus}

\newcommand{\needtocheck}[1][]{%
	\ifthenelse{\equal{#1}{}}{%
		\textcolor{blue}{[NeedToCheck]}%
	}{%
		\textcolor{blue}{[NeedToCheck: #1]}%
	}%
}

\newcommand{\todo}[1][]{%
	\ifthenelse{\equal{#1}{}}{%
		\textcolor{red}{[TODO]}%
	}{%
		\textcolor{red}{[TODO: #1]}%
	}%
}

\theoremstyle{definition}
\newtheorem{thm}{Theorem}[section]
\newtheorem*{thm*}{Theorem}
\newtheorem{defi}[thm]{Definition}
\newtheorem*{defi*}{Definition}
\newtheorem{lem}[thm]{Lemma}
\newtheorem*{lem*}{Lemma}
\newtheorem{fact}[thm]{Fact}
\newtheorem*{fact*}{Fact}
\newtheorem*{formula*}{Formula}
\newtheorem{prop}[thm]{Proposition}
\newtheorem*{prop*}{Proposition}

\newtheorem*{exm*}{Example}

\newtheorem*{rmk*}{Remark}
\newtheorem{cor}[thm]{Corollary}
\newtheorem*{cor*}{Corollary}
\newtheorem*{notation*}{Notation}
\newtheorem{question}[thm]{Question}

\DeclareMathAlphabet{\mymathbb}{U}{BOONDOX-ds}{m}{n}

\usepackage[backend=biber,style=alphabetic,sorting=nty,doi=false,isbn=false,url=false,eprint=true]{biblatex}
\renewbibmacro{in:}{}

\title{The comparability numbers and the incomparability numbers}
\author{Tatsuya Goto\thanks{Supported by JSPS KAKENHI Grant Number JP22J20021}}
\affil{Graduate School of System Informatics,\\
	Kobe University,\\
	1-1 Rokkodai, Nada-ku,\\
	657-8501 Kobe, Japan.\\
	E-mail: 202x603x@cloud.kobe-u.jp}
\date{\today}

\begin{document}
	\maketitle
	
	\begin{abstract}
		We introduce new cardinal invariants of a poset, called the comparability number and the incomparability number. We determine their value for well-known posets, such as $\omega^\omega$, $\Pow(\omega)/\mathrm{fin}$, the Turing degrees $\mathcal{D}$, the quotient algebra $\Borel(2^\omega)/\nul$, the ideals $\meager$ and $\nul$.
		Moreover, we consider these invariants for the Rudin-Keisler ordering of the nonprincipal ultrafilters on $\omega$.
		We also consider these invariants for ideals on $\omega$ and on $\omega_1$.
	\end{abstract}
	
	\section{Introduction}
	
	As cardinal invariants of a poset, the dominating number and the unbounding number are well-studied.
	In this paper, as new cardinal invariants of a poset, we introduce the comparability number and incomparability number and determine their value for well-known posets.
	
	\begin{defi}
		Let $(P, \le)$ be a poset.
		We say $F \subset P$ is a \textit{dominating family} if for every $p \in P$ there is $q \in F$ such that $p \le q$.
		We say $F \subset P$ is an \textit{unbounded family} if for every $p \in P$ there is $q \in F$ such that $q \not \le p$.
		
		Define cardinal invariants $\frakd(P, \le)$ and $\frakb(P, \le)$ as follows:
		\begin{enumerate}
			\item $\frakd(P, \le) = \min \{ \abs{F} : F \subset P \text{ dominating family} \}$,
			\item $\frakb(P, \le) = \min \{ \abs{F} : F \subset P \text{ unbounded family} \}$.
		\end{enumerate}
		We call $\frakd(P, \le)$ the \textit{dominating number} for $P$ and $\frakb(P, \le)$ the \textit{bounding number} for $P$.
	\end{defi}
	
	\begin{defi}
		Let $(P, \le)$ be a poset.
		We say $F \subset P$ is a \textit{comparable family} if for every $p \in P$ there is $q \in F$ such that either $p \le q$ or $q \le p$ holds.
		We say $F \subset P$ is an \textit{incomparable family} if for every $p \in P$ there is $q \in F$ such that both $p \not \le q$ and $q \not \le p$ holds.
		
		We define cardinal invariants $\cp(P, \le)$ and $\icp(P, \le)$ as follows:
		\begin{enumerate}
			\item $\cp(P, \le) = \min \{ \abs{F} : F \subset P \text{ comparable family} \}$,
			\item $\icp(P, \le) = \min \{ \abs{F} : F \subset P \text{ incomparable family} \}$.
		\end{enumerate}
		We call $\cp(P, \le)$ the \textit{comparability number} for $P$ and $\icp(P, \le)$ the \textit{incomparability number} for $P$.
	\end{defi}
	
	$\cp(P)$ is always defined.	
	On the other hand, $\icp(P)$ may not be defined. $\icp(P)$ is defined if and only if for all $p \in P$ there is $q \in P$ such that $p$ and $q$ are incomparable.
	This is equivalent to $\cp(P) > 1$.
	
	These cardinals are related to dominating numbers and bounding numbers: $\cp(P) \le \frakd(P), \frakd(P^*)$ and $\frakb(P), \frakb(P^*) \le \icp(P)$. Here, $P^*$ is the poset with the reverse ordering of $(P, \le)$.

	As invariants related to comparability numbers and incomparability numbers, we can consider minimal sizes of maximal antichains and maximal chains.
	
	\begin{defi}
		Let $(P, \le)$ be a poset.
		A subset $C \subset P$ is called a chain of $P$ if members of $C$ are pairwise comparable.
		Similarly, a subset $A \subset P$ is called an antichain of $P$ if members of $C$ are pairwise incomparable.
		
		Define invariants $\mmc(P)$ and $\mmac(P)$ as follows:
		\begin{enumerate}
			\item $\mmc(P) = \min \{ \abs{C} : C \subset P \text{ maximal chain} \}$, and
			\item $\mmac(P) = \min \{ \abs{A} : A \subset P \text{ maximal antichain} \}$.
		\end{enumerate}
	\end{defi}
	
	As can be easily seen, a maximal antichain of $P$ is a comparable family of $P$. So we have $\cp(P) \le \mmac(P)$.
	If $\icp(P)$ is defined, then we also have $\icp(P) \le \mmc(P)$.
	So we can draw a picture as in Figure \ref{fig:mmcandmmac}.
	
	\begin{figure}[h]
		\[
		\begin{tikzpicture}[scale=0.7]
			\node (d1) at (0, 0) {$\frakd(P)$};
			\node (b1) at (0, -5.2) {$\frakb(P)$};
			\node (d2) at (5, 0) {$\frakd(P^*)$};
			\node (b2) at (5, -5.2) {$\frakb(P^*)$};
			\node (cp) at (2.5, -1.5) {$\cp(P)$};
			\node (icp) at (2.5, -4) {$\icp(P)$};
			\node (mmac) at (2.5, 0) {$\mmac(P)$};
			\node (mmc) at (2.5, -2.7) {$\mmc(P)$};
			\draw[->] (b1) -- (d1);
			\draw[->] (b2) -- (d2);
			\draw[->] (b1) -- (icp);
			\draw[->] (b2) -- (icp);
			\draw[->] (cp) -- (d1);
			\draw[->] (cp) -- (d2);
			\draw[->] (cp) -- (mmac);
			\draw[->] (icp) -- (mmc);
		\end{tikzpicture}
		\]
		\caption{Relationships}
		\label{fig:mmcandmmac}
	\end{figure}
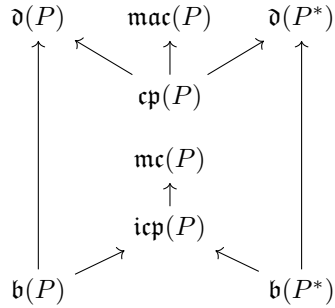
	
	The results in Table \ref{table:known} are well-known.
	
	Table \ref{table:allresults} summarizes almost all results we will prove in this paper.
	
	As results not listed in the table, in Section \ref{sec:idealsonomega}, we treat ideals on $\omega$, and in Section \ref{sec:idealonomega1}, we treat ideals on $\omega_1$.

	\begin{table}[H]
		\centering
		\begin{tabular}{l||l|l|l|l}
			\hline
			$P$ & $\frakd(P)$ & $\frakb(P)$ & $\frakd(P^*)$ & $\frakb(P^*)$   \\ \hline  \hline
			$(\omega^\omega \setminus \zero, \le^*)$ & $\frakd$ & $\frakb$ & $\frakc$ & $2$   \\ \hline
			$(\Pow(\omega)/\mathsf{fin})^-$ & $\frakc$ & $2$ & $\frakc$ & $2$   \\ \hline
			$(\Borel(2^\omega)/\meager)^-$ & $\aleph_0$ & $2$ & $\aleph_0$ & $2$       \\ \hline
			$(\Borel(2^\omega)/\nul)^-$ & $\cof(\nul)$ & $2$ & $\cof(\nul)$ & $2$        \\ \hline
			$(\nul \setminus \{\emptyset\}, \subset)$ & $\cof(\nul)$ & $\add(\nul)$ & $\frakc$ & $2$         \\ \hline
			$(\meager \setminus \{\emptyset\}, \subset)$ & $\cof(\meager)$ & $\add(\meager)$ & $\frakc$ & $2$  \\ \hline
			$\text{the Turing degrees}$ & $\frakc$ & $\aleph_1$ & $\frakc$ & $2$ \\ \hline
			$(\beta \omega \setminus \omega, \leRK)$ & $2^\frakc$ & $\frakc^+$ & depends & depends  \\ \hline
		\end{tabular}	
		\caption{\label{table:known}Known results}
	\end{table}
	
	\begin{table}[H]
		\centering
		\begin{tabular}{l||l|l|l|l}
			\hline
			$P$ & $\cp(P)$       & $\icp(P)$ & $\mmac(P)$ & $\mmc(P)$ \\ \hline  \hline
			$(\omega^\omega \setminus \zero, \le^*)$ & $\frakd$ & $\frakb$ & $\frakc$ & $\frakc$ \\ \hline
			$(\Pow(\omega)/\mathsf{fin})^-$ & $\frakr$ & $2$ & $\frakc$ \tablefootnote{This result was obtained by \cite{ahp2016}}& $\frakc$ \\ \hline
			$(\Borel(2^\omega)/\meager)^-$ & $\aleph_0$ & $2$ & ? & $\frakc$ \\ \hline
			$(\Borel(2^\omega)/\nul)^-$ & $\cof(\nul)$ & $2$ & ? & $\frakc$ \\ \hline
			$(\nul \setminus \{\emptyset\}, \subset)$ & $\cof(\nul)$ & $\add(\nul)$ & $\frakc$ & $\non(\nul)$ \\ \hline
			$(\meager \setminus \{\emptyset\}, \subset)$ & $\cof(\meager)$ & $\add(\meager)$ & $\frakc$ & $\non(\meager)$ \\ \hline
			$\text{the Turing degrees}$ & $\frakc$ & $\aleph_1$ & $\frakc$ & $\aleph_1$ \\ \hline
			$(\beta \omega \setminus \omega, \leRK)$   & depends & $\frakc^+$ or undefined & ? & $\frakc^+$ \\ \hline
		\end{tabular}	
		\caption{\label{table:allresults}Our results}
	\end{table}
	
	$\nul$ and $\meager$ denote the Lebesgue measure zero ideal for $2^\omega$ and the meager ideal for $2^\omega$, respectively.
	
	$\frakc$ denotes the cardinality of the continuum: $\frakc = 2^{\aleph_0}$.
	
	In this paper, we use famous cardinal invariants of the continuum: $\frakb$, $\frakd$, $\frakr$, $\add(\nul)$, $\add(\meager)$, $\non(\nul)$, $\non(\meager)$, $\cof(\nul)$ and $\cof(\meager)$.
	For the definition of these invariants, see \cite{blass2010combinatorial}. 
	
	Finally, we give an example of a poset with small comparability number.
	Let $P = \{0, 1\} \times \Z$ and order $P$ by
	\[(i, m) \le (j, n) \iff (i = j \land m \le n) \lor (i \ne j \land m < n).\]
	
	Then, since $\{(0, 0), (0, 1)\}$ is a maximal antichain, we have $\mmac(P) = \cp(P) = 2$.
	On the other hand, we have $\frakd(P) = \frakd(P^*) = \frakb(P) =  \frakb(P^*) = \icp(P) = \aleph_0$.
	
	\section{General lemmas}\label{sec:lemmas}
	
	The following 3 lemmas are well known and easy to see.
	
	\begin{lem}\label{lem:embedsq}
		Let $P$ be a poset. Suppose that $P$ has the following property:
		\begin{align}
			\tag{$*$} \text{If $a < b$ in $P$ then there is $c \in P$ such that $a < c < b$.}
		\end{align}
		Then $P$ embeds the set of rational numbers $\Q$.
	\end{lem}
	
	\begin{lem}\label{lem:embedsr}
		Let $P$ be a poset. Assume $P$ has the property in Lemma \ref{lem:embedsq}. Moreover, suppose that $P$ has the following property:
		\begin{align*}
			&\tag{$**$} \text{If $\seq{a_n : n \in \omega}$ is an increasing sequence of $P$ and} \\
			&\hspace{2ex} \text{$\seq{b_m : m \in \omega}$ is a decreasing sequence of $P$ and } \\
			&\hspace{4ex} \text{$(\forall n, m \in \omega)(a_n < b_m)$ holds,}\\
			&\hspace{6ex} \text{then there is $c \in P$ such that $(\forall n, m \in \omega)(a_n < c < b_m)$.}
		\end{align*}
		Then $P$ embeds the set of real numbers $\R$.
	\end{lem}
	
	\begin{lem}
		Both ($*$) and ($**$) in Lemma \ref{lem:embedsq} and \ref{lem:embedsr} are inherited by any maximal chains.
	\end{lem}
	
	\section{The cardinal invariants of $\omega^\omega$}
	
	In this section, we determine the comparability number and the incomparability number of $\omega^\omega$ as a first result.
	
	\begin{defi}
		Let $\zero$ be the set of eventually zero reals, that is,
		\[\zero = \{ x \in \omega^\omega: (\forall^\infty n)(x(n) = 0)\}.\]
		For $x, y \in \omega^\omega$, the relation $x \le^* y$ means that $(\forall^\infty n)(x(n) \le y(n))$.
		And for $x, y \in \omega^\omega$, the relation $x <^\infty y$ means that $\neg (y \le^* x)$, that is $(\exists^\infty n)(x(n) < y(n))$.
		
		Here, $(\forall^\infty n)$ and $(\exists^\infty n)$ are shortcuts for ``for all but finitely many $n$" and ``there exist infinitely many $n$" respectively.
	\end{defi}
	
	We consider the poset $(\omega^\omega \setminus \zero, \le^*)$.
	
	\begin{lem}
		$\frakb \le \icp(\omega^\omega \setminus \zero)$ and $\cp(\omega^\omega \setminus \zero) \le \frakd$ hold.
	\end{lem}
	\begin{proof}
		This is immediate from the definition.	
	\end{proof}
	
	\begin{defi}
		Define relational systems $\ICP, \frakB$ and $\BIP$ as follows:
		\begin{enumerate}
			\item $\ICP = (\omega^\omega \setminus \zero, \omega^\omega \setminus \zero, <^\infty \cap >^\infty)$.
			\item $\frakB = (\omega^\omega, \omega^\omega, <^\infty)$.
			\item $\BIP = (\IP, \IP, \{ (\mathbb{I}, \mathbb{J}) : (\exists^\infty n)(\forall k)(J_k \not \subset I_n) \})$.	
		\end{enumerate}
		Here $\IP$ is the set of all interval partition of $\omega$, $\mathbb{I} = \seq{ I_n : n \in \omega}$ and $\mathbb{J} = \seq{ J_k : k \in \omega}$.
	\end{defi}
	
	It is well-known that $\frakB$ and $\BIP$ are Tukey equivalent (for example, see \cite[Theorem 2.10]{blass2010combinatorial}).
	
	\begin{thm}
		$\icp(\omega^\omega \setminus \zero) = \frakb$ and $\cp(\omega^\omega \setminus \zero) = \frakd$ hold.
	\end{thm}
	\begin{proof}
		We construct a Tukey morphism $\ICP \to \BIP$.
		
		So we have to construct maps $\varphi \colon \omega^\omega \setminus \zero \to \IP$ and $\psi \colon \IP \to \omega^\omega \setminus \zero$ that satisfy the following condition:
		\begin{align*}
			&\text{If $x \in \omega^\omega \setminus \zero$, $\mathbb{J} = \seq{ J_k : k \in \omega} \in \IP$ satisfy} \\
			&\hspace{0.5cm}(\exists^\infty n)(\forall k)(J_k \not \subset \varphi(x)(n)) \text{ then $x <^\infty \psi(\mathbb{J})$ and $x >^\infty \psi(\mathbb{J})$}.
		\end{align*}
		Enumerate $\{n : x(n) > 0\}$ by $\{n : x(n) > 0\} = \{a^x_0 < a^x_1 < a^x_2 < \dots\}$.
		Define $\varphi$ and $\psi$ by the following way:
		\[
		\varphi_1(x)(n) = [i_n, i_{n+1}],
		\]
		where $i_0 = 0$ and $i_{n+1}$ are such that the interval $[i_n, i_{n+1})$ contains at least 3 points of the form $a^x_j$ and for all $a \le i_n$, $x(a) \le i_{n+1}$
		and
		\[
		\psi(\mathbb{J})(n) = \begin{cases}
			\min J_{k+2} & \text{ (if $n \in J_k$ and $n = \min J_k$)} \\
			0 & \text{ (if $n \in J_k$ and $n > \min J_k$)}.
		\end{cases}
		\]
		We first show that $x >^\infty \psi(\mathbb{J})$.
		Take $n_0 \in \omega$ arbitrarily. Then we can take $n > n_0$ such that $(\forall k)(J_k \not \subset \varphi_1(x)(n))$.
		Let $I_n = \varphi_1(x)(n)$.
		Then we take $k$ such that $I_n \cap J_k \ne \emptyset$. 
		Note that the number of such $k$ is less than or equal to $2$.
		But we have at least $3$ points $a^x_i$ in $I_n$.
		So we can take $a^x_i \in I_n$ that is not the leftmost point of intervals in $\mathbb{J}$.
		We have $a^x_i \ge a^x_{3n} \ge 3n > n_0$, $x(a^x_{i}) > 0$ and $\psi(\mathbb{J})(a^x_i) = 0$.
		Thus we have $x >^\infty \psi(\mathbb{J})$.
		
		We next prove $x <^\infty \psi(\mathbb{J})$.
		Let $k_0 \in \omega$.
		By $(\exists^\infty n)(\forall k)(J_k \not \subset \varphi(x)(n))$, we can take $n$ such that $i_n > j_{k_0}$ and $(\forall k)(J_k \not \subset I_n)$.
		Let $k$ be such that $i_n \in J_k$.
		Then $j_k \le i_n$ and $i_{n+1} < j_{k+2}$ since there are at most $2$ intervals in $\mathbb{J}$ touching $I_n$.
		By the choice of $i_{n+1}$, we have $x(j_k) \le i_{n+1} < j_{k+2}$. Thus $x(j_k) < \psi(\mathbb{J})(j_k)$.
		Also, by $i_n \in J_k$, we have $i_n < j_{k+1}$.
		So $j_{k_0} < i_n < j_{k+1}$. Thus $k_0 \le k$.
		Thus we have proved $x <^\infty \psi(\mathbb{J})$.
	\end{proof}
	
	\begin{thm}\label{thm:mmcomegaomega}
		$\mmc(\omega^\omega \setminus \zero) = \frakc$.
	\end{thm}
	
	\begin{proof}
		Every maximal chain of $\omega^\omega \setminus \zero$ satisfies the assumption in Lemma \ref{lem:embedsr}.
	\end{proof}
	
	The following theorem was obtained through private communication with Jorge Antonio Cruz Chapital.
	
	\begin{thm}\label{thm:mmacomegaomega}
		$\mmac(\omega^\omega\setminus \zero) = \frakc$.
	\end{thm}
	\begin{proof}
		Let $\mathcal{A}$ be a maximal antichain of $\omega^\omega\setminus \zero$.
		Fix $\psi \in \mathcal{A}$.
		Let $X = \{ n \in \omega : \psi(n) > 0 \}$.
		Take a family $\seq{ (A_\alpha, B_\alpha) : \alpha < \frakc }$ of pairs of elements in $[X]^\omega$ such that $A_\alpha \cap B_\alpha = \emptyset$ for every $\alpha$ and $A_\alpha \cup B_\alpha$ and $A_\beta \cup B_\beta$ are almost disjoint for every distinct $\alpha$ and $\beta$.
		For $\alpha < \frakc$, we define $g_\alpha$ by
		\[
		g_\alpha(n) = \begin{cases}
			\psi(n)+1 & \text{ (if $n\in A_\alpha$)} \\
			\psi(n)-1 & \text{ (if $n\in B_\alpha$)} \\
			\psi(n) & \text{ (otherwise)}.
		\end{cases}
		\]
		Define two sets $Y_0, Y_1 \subset \frakc$ by
		\begin{align*}
			Y_0 &= \{ \alpha < \frakc : (\exists f \in \mathcal{A}) (g_\alpha \le^* f) \} \\
			Y_1 &= \{ \alpha < \frakc : (\exists f \in \mathcal{A}) (f \le^* g_\alpha) \}
		\end{align*}
		Since $Y_0 \cup Y_1 = \frakc$, we have either $\abs{Y_0} = \frakc$ or $\abs{Y_1} = \frakc$.
		
		Consider the case $\abs{Y_0} = \frakc$.
		For each $\alpha \in Y_0$, take $f_\alpha \in \mathcal{A}$ such that $g_\alpha \le^* f_\alpha$.
		Then for each $\alpha \in Y_0$, we have $\{ n : f_\alpha(n) < \psi(n) \} \subset^* B_\alpha$.
		Therefore, for distinct $\alpha$ and $\beta$, we have $\{ n : f_\alpha(n) < \psi(n) \}$ and $\{ n : f_\beta(n) < \psi(n) \}$ are almost disjoint.
		Thus, we have proved $f_\alpha \ne f_\beta$ whenever $\alpha$ and $\beta$ are distinct.
		So it holds that $\abs{\mathcal{A}} = \frakc$.
		
		The proof is similar for the case $\abs{Y_1} = \frakc$.
	\end{proof}

	\section{The cardinal invariants of Boolean algebras}\label{sec:boolean}
	
	In this section, we deal with (in)comparability numbers of Boolean algebras.
	We write the Boolean operations as $+, \cdot$ and $(-)^\comp$: join, meet and complementation.
	Moreover, $0$ and $1$ mean the minimum and maximum elements of the Boolean algebra.
	
	\begin{defi}
		Let $B$ be a Boolean algebra. Then we define $B^-$ by
		\[
		B^- = B \setminus \{ 0, 1 \}.
		\]
	\end{defi}
	
	\begin{lem}
		Let $B$ be a Boolean algebra that is not equal to $\{0, 1\}$.
		Then $\icp(B^-) = 2$.
	\end{lem}
	\begin{proof}
		Take an element $b \in B \setminus \{0, 1\}$.
		Then $F = \{ b, b^\comp \}$ satisfies
		\[
		(\forall x \in B^-)(\exists y \in F)(x \not \le y \AND y \not \le x).
		\]
		In order to show this, let $x \in \Pow(\omega)^-$.
		Assume that $x \le b$ or $b \le x$.
		In either case, we can easily show that both $x \not \le b^\comp$ and $b^\comp \not \le  x$.
	\end{proof}
	
	\begin{defi}
		Let $B$ be a Boolean algebra and $D$ be a subset of $B \setminus \{0\}$.
		We say $D$ is \textit{weakly dense set} of $B$ if for all $b \in B \setminus \{0\}$ there is $d \in D$ such that $d \le b$ or $d \le b^\comp$.
		Put 
		\[
		\WD(B) = \min \{\abs{D} : D \text{ is weakly dense set of } B \}
		\]
	\end{defi}
	
	\begin{lem}\label{lem:wdbisinfinite}
		If $B$ is an atomless Boolean algebra, then $\WD(B)$ is infinite.
	\end{lem}
	\begin{proof}
		Suppose that $D$ is a finite weakly dense set.
		Let $D'$ be the set of finite meets of elements of $D$ that is not equal to $0$.
		Let $D''$ be the set of minimal elements of $D'$.
		Then $D''$ is a finite weakly dense set such that for every distinct $d, e \in D''$, we have $d \cdot e = 0$.
		We may assume that given $D$ has this property.
		
		Enumerate $D$ as $D = \{ d_0, \dots, d_{n-1} \}$.
		For each $i < n$, take an element $e_i$ such that $0 < e_i < d_i$. We can take these elements since $B$ is atomless.
		Put $b = e_0 + \dots + e_{n-1}$.
		Then we have $d_i \not \le b$ and $d_i \not \le b^\comp$ for every $i < n$. This is a contradiction.		
	\end{proof}
	
	\begin{lem}\label{lem:cpequalswd}
		Let $B$ be a Boolean algebra.
		Then we have $\cp(B^-) \le 2 \WD(B)$ and $\WD(B) \le 2 \cp(B^-)$.
		In particular, if either $\cp(B^-)$ or $\WD(B)$ is infinite, then we have $\cp(B^-) = \WD(B)$.
	\end{lem}
	\begin{proof}
		First we show $\WD(B) \le 2 \cp(B^-)$.
		Let $C$ be a comparable family of $B^-$ of size $\cp(B^-)$.
		Then $C' = C \cup \{ c^\comp : c \in C \}$ is a weakly dense set of $B$.
		Now we have $\abs{C'}\le 2 \abs{C} = 2 \cp(B^-)$.
		So $\WD(B) \le 2 \cp(B^-)$.
		
		Next we show $\cp(B^-)\le 2 \WD(B)$.
		Let $D$ be a weak dense family of $B$ of size $\WD(B^-)$.
		Then $D' = D \cup \{ d^\comp : d \in D \}$ is a comparable family of $B^-$.
		Now we have $\abs{D'} \le 2 \abs{D} = 2 \WD(B)$.
		So $\cp(B^-) \le 2 \WD(B)$.
	\end{proof}
	
	\section{The cardinal invariants of $\Pow(\omega)/\mathsf{fin}$}
	
	\begin{cor}
		$\cp((\Pow(\omega))/\mathrm{fin})^-) = \frakr$. \qed
	\end{cor}
	\begin{proof}
		This follows from Lemma \ref{lem:cpequalswd}.
	\end{proof}
	
	The following fact was discovered by G. Campero-Arena, J. Cancino, M. Hrušák and F. E. Miranda-Perea.
	
	\begin{fact}[{{\cite[Corollary 2.4]{ahp2016}}}]
		$\mmac((\Pow(\omega)/\fin)^-) = \frakc$.
	\end{fact}
	
	\section{The cardinal invariants of the Cohen algebra and the random algebra}
	
	\begin{cor}
		$\cp((\Borel(2^\omega)/\nul)^-) = \cof(\nul)$.
	\end{cor}
	\begin{proof}
		This follows from Lemma \ref{lem:cpequalswd} and Theorem 1 in \cite{burke1989weakly} that states that $\WD(\Borel(2^\omega)/\nul) = \cof(\nul)$.
	\end{proof}

	\begin{prop}
		$\mmc((\Borel(2^\omega)/\nul)^-) = \mmc((\Borel(2^\omega)/\meager)^-) = \frakc$.
	\end{prop}
	\begin{proof}
		This follows from the fact that the above 2 Boolean algebras are $\sigma$-complete and lemmas in Section \ref{sec:lemmas}.
	\end{proof}

	\section{The cardinal invariants of the ideal $\nul$}\label{sec:nullideal}
	
	In this section, we determine the values $\cp(\nul \setminus \{\emptyset\})$ and $\icp(\nul \setminus \{\emptyset\})$.
	
	\begin{fact}[{{\cite[Lemma 1.3.23]{bartoszynski1995set}}}]\label{fact:indep}
		Suppose that $\seq{ a_n : n \in \omega }$ is a sequence of reals in $(0, 1)$.
		Then there is a sequence $\seq{A_n : n \in \omega}$ of open sets of $2^\omega$ such that it is independent in the sense of probability theory and $\mu(A_n) = a_n$.
	\end{fact}
	
	\begin{lem}\label{lem:cofnull}
		If $\mathcal{F} \subset \nul$ is a family of size less than $\cof(\nul)$, then there is a $B \in \nul$ such that for all $A \in \mathcal{F}$ we have $\abs{B \setminus A} = \frakc$.
	\end{lem}
	\begin{proof}
		This proof is based on \cite[Lemma 2.3.3]{bartoszynski1995set}.
		Let $\mathcal{C} = \{ S \in (\omega^{<\omega})^\omega : \sum \frac{S(n)}{(n+1)^2} < \infty \}$.
		And for $S, S' \in \mathcal{C}$, define $S \le S'$ by $S \le S' \iff (\forall^\infty)(S(n) \le S'(n))$.
		It is known that $\mathcal{C}$ and $\nul$ are Tukey equivalent.
		So it suffices to show that $\mathcal{C} \le_\mathrm{T} (\nul, \nul, \subset^*)$.
		Here $A \subset^* B$ means that $\abs{A \setminus B} < \frakc$.
		
		We have to construct $\varphi, \psi$ such that $\varphi \colon \mathcal{C} \to \nul$, $\psi \colon \nul \to \mathcal{C}$ and $(\forall S \in \mathcal{C})(\forall G \in \nul)(\varphi(S) \subset^* G \rightarrow S \le^* \psi(G))$ hold.
		
		By Fact \ref{fact:indep}, fix a sequence $(G_{n,i} : n, i \in \omega)$ of open sets such that
		$G_{n,i}$ has measure $1/(n+1)^2$ and the sequence $(G_{n,i} : n \in \omega)$ is independent for every $i \in \omega$.
		
		Define $\varphi \colon \mathcal{C} \to \nul$ by
		\[
		\varphi(S) = \bigcap_{m \in \omega} \bigcup_{n \ge m} \bigcup_{i \in S(n)} G_{n,i}.
		\]
		
		For $G \in \nul$, fix a perfect set $K^G$ of positive measure such that $G \cap K^G = \emptyset$.
		We can assume that $K^G \cap U \ne \emptyset$ implies $\mu(K^G \cap U) > 0$ for every basic open set $U$.
		Let $(U_n : n \in \omega)$ be an enumeration of all basic open sets $U$ such that $K^G \cap U \ne \emptyset$.
		Put
		\[
		A^G_{n,i} = \{ j \in \omega : K^G \cap U_n \cap G_{i,j} = \emptyset \}.
		\]
		Then we can show that $A^G_{n,i} \in \mathcal{C}$.
		Take a slalom $S \in \mathcal{C}$ such that $(A^G_{n,i} : i \in \omega) \le S$ for all $n \in \omega$.
		Define $\psi(G)$ by putting $\psi(G)$ be this $S$.
		
		We have to show $(\forall S \in \mathcal{C})(\forall G \in \nul)(\varphi(S) \subset^* G \rightarrow S \le^* \psi(G))$.
		Fix $S \in \mathcal{C}$ and $G \in \nul$.
		Then we have $\abs{\varphi(S) \cap K^G} \le \abs{\varphi(S) \setminus G} < \frakc$.
		Since $\varphi(S) \cap K_G$ is a Borel set, we have $\abs{\varphi(S) \cap K^G} \le \aleph_0$ by the perfect set theorem.
		
		We have
		\[
		\bigcap_{m \in \omega} (K^G \cap \bigcup_{n \ge m} \bigcup_{i \in S(n)} G_{n,i}) \cap \bigcap_{x \in \varphi(S) \cap K^G} (K^G \setminus \{x\}) = \emptyset.
		\]
		So by the Baire category theorem applied to the space $K^G$, at least one term in the above intersection is not dense in $K^G$.
		So, there is a $n_0 \in \omega$ such that $K^G \cap \bigcup_{n \ge n_0} \bigcup_{i \in S(n)} G_{n,i}$ is not dense in $K^G$.
		So we can take $m \in \omega$ such that $K^G \cap U_m \cap \bigcup_{n \ge n_0} \bigcup_{i \in S(n)} G_{n,i} = \emptyset$.
		Then we have $(\forall n \ge n_0)(\forall i \in S(n))(K^G \cap U_m \cap G_{n,i} = \emptyset)$.
		So we have $(\forall^\infty n)(S(n) \subset A^G_{m,n} \subset \psi(G)(n))$.
		Thus $S \le \psi(G)$ holds.
	\end{proof}
	
	\begin{thm}\label{thm:cpnull}
		$\cp(\nul \setminus \{\emptyset\}) = \cof(\nul)$.
	\end{thm}
	\begin{proof}
		It is clear that $\cp(\nul \setminus \{\emptyset\}) \le \cof(\nul)$.
		So it suffices to show $\cof(\nul) \le \cp(\nul \setminus \{\emptyset\})$.
		
		Suppose $\kappa < \cof(\nul)$ and take $\mathcal{F} \subset \nul \setminus \{\emptyset\}$ of size $\kappa$.
		Then by Lemma \ref{lem:cofnull}, we can take $B \in \nul$ such that for all $A \in \mathcal{F}$ we have $\abs{B \setminus A} = \frakc$.
		For each $A \in \mathcal{F}$, fix an element $x_A \in A$.
		Put $B' = B \setminus \{ x_A : A \in \mathcal{F} \}$.
		Then $B'$ is a incomparable with all $A \in \mathcal{F}$, since $x_A \in A \setminus B'$ and $\abs{B \setminus A} = \frakc$ and $\abs{B \setminus B'} < \frakc$.
	\end{proof}
	
	\begin{thm}\label{thm:icpnull}
		$\icp(\nul \setminus \{\emptyset\}) = \add(\nul)$.
	\end{thm}
	\begin{proof}
		It is clear that $\add(\nul) \le \icp(\nul \setminus \{\emptyset\})$.
		So we have to show that $\icp(\nul \setminus \{\emptyset\}) \le \add(\nul)$.
		Take a sequence $\seq{A_\alpha : \alpha < \add(\nul)}$ of null sets whose union is not null.
		Put $B_\alpha = A_\alpha \setminus \bigcup_{\beta < \alpha} A_\beta$.
		Then $\mathcal{F} = \{ B_\alpha : \alpha <  \add(\nul) \} \setminus \{ \emptyset \}$ is an incomparable family.
		To prove this, let $C \in \nul \setminus \{ \emptyset \}$.
		Since we have $C \in \nul$ and $\bigcup \mathcal{F} \not \in \nul$, there is an $\alpha < \add(\nul)$ such that $B_\alpha \not \subset C$.
		If $C \not \subset B_\alpha$ holds, then we are done.
		If $C \subset B_\alpha$ holds, then we take another piece $B_\beta$. Then $C$ and $B_\beta$ are disjoint nonempty sets, in particular, they are incomparable.
	\end{proof}
	
	\begin{prop}\label{prop:mmcnul}
		$\mmc(\nul) = \non(\nul)$.
	\end{prop}
	\begin{proof}
		We first prove $\mmc(\nul) \le \non(\nul)$.
		Take a non-null set $X = \{ x_\alpha : \alpha< \non(\nul) \}$.
		For each $\alpha$, set $X_\alpha = \{ x_\beta : \beta < \alpha \}$.
		Then $\{ X_\alpha : \alpha < \non(\nul) \}$ is a maximal chain.
		
		We next prove $\non(\nul) \le \mmc(\nul)$.
		Take a maximal chain $\mathcal{C}$ of $\nul$.
		We have $\bigcup \mathcal{C} \not \in \nul$. In fact, otherwise, we can extend the chain $\mathcal{C}$ upwards.
		Set $X = \bigcup \mathcal{C}$.
		
		For each $x \in X$, put
		\begin{align*}
			\mathcal{L}_x &= \{ C \in \mathcal{C} : x \not \in C \}, \\
			\mathcal{R}_x &= \{ D \in \mathcal{C} : x \in D \}.
		\end{align*}
		Then we have $\mathcal{L}_x \cup \mathcal{R}_x = \mathcal{C}$ (disjoint union) and for every $C \in \mathcal{L}_x$ and $D \in \mathcal{R}_x$, $C \subset D$.
		We put $D_x = \bigcap \mathcal{R}_x$. By maximality of $\mathcal{C}$, we have $D_x \in \mathcal{C}$.
		In addition, it can be easily shown that the map $X \ni x \mapsto D_x \in \mathcal{C}$ is injective.
		
		Therefore, we have $\non(\nul) \le \abs{X} \le \abs{\mathcal{C}}$. So it holds that $\non(\nul) \le \mmc(\nul)$.
	\end{proof}
	
	\begin{prop}\label{prop:mmacnul}
		$\mmac(\nul \setminus \{ \emptyset \}) = \frakc$.
	\end{prop}
	\begin{proof}
		This proof is based on \cite[Proposition 2.3]{ahp2016}.
		Clearly, $\{ \{x\} : x \in 2^\omega \}$ is a maximal antichain of $\nul \setminus \{ \emptyset \}$. So we have $\mmac(\nul \setminus \{ \emptyset \}) \le \frakc$.
		
		Let $A, A' \in \nul$ be such that $\abs{A} = \abs{A'} = \frakc$ and $A \cap A' = \emptyset$.
		To prove $\mmac(\nul \setminus \{ \emptyset \}) \ge \frakc$, let $\mathcal{A}$ be an antichain of size ${<} \frakc$.
		Let $\mathcal{C}$ be the closure of $\mathcal{A} \cup \{A, A'\}$ under the operation of finite unions, finite intersections and taking difference sets.
		Since we have $\abs{\mathcal{C}} < \frakc$, which is the density of each of $\Pow(A) \setminus \{ \emptyset \}$ and $\Pow(A') \setminus \{ \emptyset \}$, we can take $C_0 \subset A'$ and $C_1 \subset A$ nonempty such that
		\begin{align}
			\neg (\exists B \in \mathcal{C}\setminus \{ \emptyset \})(B \subset C_0 \text{ or } B \subset C_1). \tag{$*$}\label{eq:nondense}
		\end{align}
		Set $D = (A \setminus C_1) \cup C_0$.
		
		We claim $D \not \in \mathcal{A}$. If $D \in \mathcal{A}$ holds, then we have $D \setminus A = C_0 \in \mathcal{C} \setminus \{ \emptyset \}$, which contradicts (\ref{eq:nondense}).
		Fix $X \in \mathcal{A}$ arbitrary.
		We next claim $D$ and $X$ are incomparable.
		If $D \subset X$, then $A \setminus X \subset A \setminus D = C_1$ holds. This contradicts  $A \setminus X \in \mathcal{C} \setminus \{ \emptyset \}$ and (\ref{eq:nondense}).
		If $X \subset D$, then $X \setminus A \subset D \setminus A = C_0$ holds. This contradicts  $X \setminus A \in \mathcal{C} \setminus \{ \emptyset \}$ and (\ref{eq:nondense}).
		
		Therefore, we have $\mathcal{A} \cup \{ D \}$ is bigger antichain than $\mathcal{A}$. So $\mathcal{A}$ is not maximal.
	\end{proof}
	
	\section{The cardinal invariants of the ideal $\meager$}\label{sec:meagerideal}
	
	In this section, we determine the values $\cp(\meager \setminus \{\emptyset\})$ and $\icp(\meager \setminus \{\emptyset\})$ by the same method as in the previous section.
	
	\begin{defi}
		For an interval partition $\mathbb{I} = (I_n : n \in \omega)$ and a real $x \in 2^\omega$, we put
		\[\Match(x, \mathbb{I}) = \{ y \in 2^\omega : (\exists^\infty n)(x \upharpoonright I_n = y \upharpoonright I_n) \}.\]
	\end{defi}
	
	\begin{fact}\label{fact:match}
		\begin{enumerate}
			\item $\Match(x, \mathbb{I})$ is a comeager set for every interval partition $\mathbb{I} = (I_n : n \in \omega)$ and every real $x \in 2^\omega$.
			\item \cite[Theorem 5.2]{blass2010combinatorial} For every meager set $A \subset 2^\omega$, there is an interval partition $\mathbb{I} = (I_n : n \in \omega)$ and a real $x \in 2^\omega$ such that $A \cap \Match(x, \mathbb{I}) = \emptyset$.
		\end{enumerate}
	\end{fact}
	
	\begin{lem}\label{lem:match}
		Let $\mathbb{I} = (I_n : n \in \omega), \mathbb{J} = (J_k : k \in \omega) \in \IP$ and $x, y \in 2^\omega$. Suppose that $\abs{J_k} \ge 2$ for every $k$.
		Then the following are equivalent.
		\begin{enumerate}
			\item $\Match(x, \mathbb{I}) \not \subset \Match(y, \mathbb{J})$.
			\item The set $\Match(x, \mathbb{I}) \setminus \Match(y, \mathbb{J})$ has size $\frakc$.
			\item $(\exists^\infty n)(\forall k)(J_k \not \subset I_n \text{ or } x \upharpoonright J_k \ne y \upharpoonright J_k)$
		\end{enumerate}
	\end{lem}
	\begin{proof}
		This lemma is an improvement of \cite[Proposition 5.3]{blass2010combinatorial}.
		That (2) implies (1) is clear. Moreover, that (1) implies (3) is not difficult.
		So we shall show (3) implies (2).
		Take an infinite set $A \subset \omega$ such that
		\[
		(\forall n \in A)(\forall k)(J_k \not \subset I_n \text{ or } x \upharpoonright J_k \ne y \upharpoonright J_k). \tag{$*$}
		\]
		We can assume that 
		\[(\forall n)(\{n, n+1\} \not \subset A). \tag{$**$} \]
		Let
		\begin{align*}
			A' &= \{ n \in A : \text{$n$ is $2l$-th element of $A$ for some $l$}\} \\
			A'' &= \{ n \in A : \text{$n$ is $(2l+1)$-th element of $A$ for some $l$}\}
		\end{align*}

		For $z \in 2^\omega$, we put
		\[
		w_z(m) = \begin{cases}
			x(m) & \text{(if $m \in \bigcup_{n \in A'} I_n$)} \\
			z(l) & \text{(if $m$ is $l$-th element of $\bigcup_{n\in A''} \{\min I_n\}$)} \\
			1 - y(m) & \text{otherwise}
		\end{cases}
		\]
		
		Since $(\forall n \in A)(w_z \upharpoonright I_{n} = x \upharpoonright I_{n})$ holds, we have $w_z \in \Match(x, \mathbb{I})$.
		
		We now prove that $w_z \not \in \Match(y, \mathbb{J})$.
		In order to prove it, let $k \in \omega$.
		
		Suppose that there is an $n \in \omega$ such that $J_k \subset I_n$.
		If $n \in A'$ then we have $w_z \upharpoonright J_k = x \upharpoonright J_k \ne y \upharpoonright J_k$ by ($*$).
		If $n \not \in A'$, then we have either $n \in A''$ or 	$n \in \omega \setminus A$.
		In the former case, $w_z(m) \ne y(m)$ for $m \in J_k \setminus \{\min I_n\}$. Here we used $\abs{J_k} \ge 2$. In the latter case, we have $w_z(m) = 1 - y(m) \ne y(m)$ for every $m \in J_k$.
		
		Suppose that for every $n \in \omega$ we have $J_k \not \subset I_n$.
		Then $J_k$ touches greater than or equal to $2$ intervals in $\mathbb{I}$.
		At least one $n$ of them satisfies $n \not \in A$ by ($**$). Fix such an $n$.
		For $m \in J_k \cap I_n$, we have $w_z(m) = 1 - y(m) \ne y(m)$.
		So we have proved $(\forall k)(w_z \upharpoonright J_k \ne y \upharpoonright J_k)$.
		Thus, we have $w_z \not \in \Match(y, \mathbb{J})$.
		
		Since $w_z \ (z \in 2^\omega)$ are distinct reals, we are done.
	\end{proof}
	
	\begin{lem}\label{lem:cofmeager}
		If $\mathcal{F} \subset \meager$ is a family of size less than $\cof(\meager)$, then there is a $B \in \meager$ such that for all $A \in \mathcal{F}$ we have $\abs{B \setminus A} = \frakc$.
	\end{lem}
	\begin{proof}
		For $A \in \mathcal{F}$, take $x_A \in 2^\omega$ and $\mathbb{I}_A \in \IP$ such that $A \cap \Match(x_A, \mathbb{I}_A) = \emptyset$.
		Since each $\Match(x_A, \mathbb{I}_A)^\comp$ is meager set, by the definition of $\cof(\meager)$, we can take $B \in \meager$ such that $B \setminus \Match(x_A, \mathbb{I}_A)^\comp \ne \emptyset$.
		Take $y \in 2^\omega$ and $\mathbb{J} \in \IP$ such that $B \cap \Match(y, \mathbb{J}) = \emptyset$.
		We can assume that $\abs{J_k} \ge 2$ for every $k \in \omega$.
		Then we have $\Match(y, \mathbb{J})^\comp \setminus \Match(x_A, \mathbb{I}_A)^\comp \ne \emptyset$. That is, we have $\Match(x_A, \mathbb{I}_A) \setminus \Match(y, \mathbb{J}) \ne \emptyset$.
		So by Lemma \ref{lem:match}, $\Match(x_A, \mathbb{I}_A) \setminus \Match(y, \mathbb{J})$ has size $\frakc$.
		Now put $C =  \Match(y, \mathbb{J})^\comp$.
		Then $C$ is meager and for all $A \in \mathcal{F}$, we have $\abs{C \setminus A} \ge \abs{\Match(x_A, \mathbb{I}_A) \setminus \Match(y, \mathbb{J})} \ge \frakc$.
		So $C$ witnesses the lemma.
	\end{proof}
	
	\begin{thm}
		$\cp(\meager \setminus \{\emptyset\}) = \cof(\meager)$.
	\end{thm}
	\begin{proof}
		This theorem can be shown by the same proof on Theorem \ref{thm:cpnull} using Lemma \ref{lem:cofmeager} instead of Lemma \ref{lem:cofnull}.
	\end{proof}
	
	\begin{thm}\label{thm:icpmeager}
		$\icp(\meager \setminus \{\emptyset\}) = \add(\meager)$.
	\end{thm}
	\begin{proof}
		This can be shown by the same argument of Theorem \ref{thm:icpnull}.
	\end{proof}

	\begin{prop}\label{prop:mmcmeager}
		$\mmc(\meager) = \non(\meager)$ and $\mmac(\meager \setminus \{ \emptyset \}) = \frakc$ hold.
	\end{prop}
	\begin{proof}
		This proposition can be shown by the same argument of Proposition \ref{prop:mmcnul} and \ref{prop:mmacnul}.
	\end{proof}
	
	\section{The cardinal invariants of Turing degrees}
	
	In this section, we deal with the Turing degrees.
	Let $\TD^+$ denote the poset of all incomputable Turing degrees.
	
	The following fact is well-known.
	
	\begin{fact}
		$\mmac(\TD^+) = \frakc$ and $\mmc(\TD^+) = \aleph_1$.
	\end{fact}
	\begin{proof}
		Since $\TD^+$ is $\sigma$-upward directed, we have that $\mmc(\TD^+)$ is uncountable.
		Moreover, since each downward cone of $\TD^+$ is countable, we have $\mmc(\TD^+) = \aleph_1$.
		
		Since there are $\frakc$ many minimal elements in $\TD^+$, we have $\mmac(\TD^+) \le \frakc$.
		Suppose that there is a maximal antichain $A$ of size less than $\frakc$ of $\TD^+$.
		Then $A{\downarrow} = \{ x \in \TD^+ : (\exists y \in A) (x \le y)\}$ has also size less than $\frakc$.
		Thus, we can take a minimal element that does not belong to $A{\downarrow}$. This contradicts maximality of $A$.
	\end{proof}
	
	Using the above fact, we prove the following proposition.
	
	\begin{prop}
		$\cp(\TD^+) = \frakc$ and $\icp(\TD^+) = \aleph_1$.
	\end{prop}
	\begin{proof}
		To show $\cp(\TD^+) = \frakc$, we fix a comparable family $\mathcal{A} = (A_\alpha : \alpha < \kappa)$.
		Put $\mathcal{A}' = \{ A : A \leT A_\alpha \text{ for some } \alpha \}$.
		Since every downward cone in $\mathcal{D}$ is countable, we have $\abs{\mathcal{A}'} = \kappa$.
		Fix $B \subset \omega$ arbitrarily.
		Then we can find $\alpha < \kappa$ such that $A_\alpha \leT B$ or $B \leT A_\alpha$. In either case, we have $(\exists A \in \mathcal{A}')(A \leT B)$.
		So $\mathcal{A'}$ satisfies $(\forall B)(\exists A \in \mathcal{A}')(A \leT B)$.
		So $\mathcal{A'}$ is a coinitial family.
		But in the poset of Turing degrees, there are continuum many minimal elements.
		So we have $\cp(\TD^+) \ge \frakc$.
		
		Since the poset of Turing degrees is $\sigma$-upward directed, we have $\icp(\TD^+) \ge \frakb(\TD^+) \ge \aleph_1$.
		
		By the previous fact, we have $\icp(\TD^+) \le \mmc(\TD^+) \le \aleph_1$.
	\end{proof}
	
	\section{The cardinal invariants of the Rudin--Keisler ordering}\label{sec:rudinkeisler}
	
	In this section, we will focus on the Rudin--Keisler ordering on the set of nonprincipal ultrafilters on $\omega$.
	
	For the definition and basic properties of Rudin--Keisler ordering, see \cite{halbeisen2012combinatorial}.
	
	\begin{prop}
		$\frakd(\beta \omega \setminus \omega, \leRK) = 2^\frakc$.
	\end{prop}
	\begin{proof}
		Take a dominating family $D$ of $(\beta \omega \setminus \omega, \leRK)$.
		Then we have $\bigcup_{p \in D} p{\downarrow} = \beta \omega \setminus \omega$, where $p{\downarrow}$ is the downward cone below $p$, whose size is $\le \frakc$.
		So we have $2^\frakc \le \frakc \cdot \abs{D}$. Therefore we have $\abs{D} = 2^\frakc$.
	\end{proof}
	
	The next lemma is well-known.
	
	\begin{lem}\label{lem:boundingofrk}
		$\frakb(\beta \omega \setminus \omega, \leRK) \ge \frakc^+$.
	\end{lem}
	\begin{proof}
		Let $(p_\alpha : \alpha < \frakc)$ be a sequence of elements in $\beta \omega \setminus \omega$.
		We have to show that there is an upper bound of these $p_\alpha$'s.
		Take an independent family $I = \{ f_\alpha : \alpha < \frakc \}$ of functions from $\omega$ into $\omega$ of size $\frakc$.
		By independence, the set
		\[
		\{ f_{\alpha}^{-1}(A) : \alpha < \frakc, A \in p_\alpha \}
		\]
		has the strong finite intersection property.
		So there is an ultrafilter $q$ that extends this set.
		This $q$ is above all $p_\alpha$'s.
	\end{proof}
	
	$\frakb(\beta \omega \setminus \omega, \geRK)$ depends on models of set theory.
	If Near Coherence of Filters (NCF) holds, then $\frakb(\beta \omega \setminus \omega, \geRK) > 2$, but otherwise $\frakb(\beta \omega \setminus \omega, \geRK) = 2$.
	
	\begin{prop}
		Assume there exist $2^\frakc$ many Ramsey ultrafilters.
		Then we have $\cp(\beta \omega \setminus \omega, \leRK) = 2^\frakc$.
	\end{prop}
	\begin{proof}
		Take a comparable family $C \subset \beta \omega \setminus \omega$ of size less than $2^\frakc$.
		Set $C' = \{ p \in \beta \omega \setminus \omega : (\exists q \in C)(p \leRK q) \}$.
		Then $C'$ must contain all Ramsey ultrafilters.
		But the size of $C'$ is less than $2^\frakc$ because every downward cone is size $\le \frakc$.
		This contradicts our assumption.
	\end{proof}
	
	\begin{prop}
		In the Miller model over a model of $\mathsf{GCH}$, we have $\frakd(\beta \omega \setminus \omega, \geRK) \le \frakc$. In particular, $\cp(\beta \omega \setminus \omega, \leRK) \le \frakc$.
	\end{prop}
	\begin{proof}
		Note that in the model, NCF holds and there are exactly $\frakc$ many P-points.
		So the set of all P-points is a dominating family of size $\frakc$ of the poset $(\beta \omega \setminus \omega, \geRK)$.
		
		To show this, take an arbitrary ultrafilter $p$.
		And take a P-point $q$.
		By NCF, there is $r \leRK p, q$.
		Since the property being a P-point is downward closed, $r$ is also a P-point.
		So there is a P-point which is below $p$.
	\end{proof}
	
	\begin{prop}\label{prop:mmcofrk}
		$\mmc(\beta \omega \setminus \omega, \leRK) = \frakb(\beta \omega \setminus \omega, \leRK) = \frakc^+$.
	\end{prop}
	\begin{proof}
		Take a maximal chain $C$ of $\beta \omega \setminus \omega$.
		The size of $C$ is less than or equal to $\frakc^+$ since each downward cone has size $\le \frakc$.
		Therefore we have $\mmc(\beta \omega \setminus \omega, \leRK) \le \frakc^+$.
		
		So combining this fact and Lemma \ref{lem:boundingofrk}, we have
		\[
		\frakc^+ \le \frakb(\beta \omega \setminus \omega, \leRK) \le \mmc(\beta \omega \setminus \omega, \leRK) \le \frakc^+. \qedhere
		\]
	\end{proof}
	
	\begin{prop}
		If $\icp(\beta \omega \setminus \omega, \leRK)$ is defined, then $\icp(\beta \omega \setminus \omega, \leRK) = \frakc^+$.
	\end{prop}
	\begin{proof}
		This follows from Proposition \ref{prop:mmcofrk}.
	\end{proof}
	
	It is a longstanding problem for $\ZFC$ to prove that for every $p \in \beta \omega \setminus \omega$ there is $q \in \beta \omega \setminus \omega$ such that $p$ and $q$ are incomparable. In other words, we don't know that $\ZFC$ proves $\cp(\beta \omega \setminus \omega) > 1$.
	
	\section{The cardinal invariants of ideals on $\omega$}\label{sec:idealsonomega}
	
	In this section, we consider the comparability numbers and incomparability numbers of the ideals on $\omega$.
	
	For an ideal $I$ on $\omega$, recall that the additivity of $\I$, $\add^*(\I)$ is defined to be the minimal cardinality of $\mathcal{A} \subset \I$ such that for every $B \in \I$ there is $A \in \mathcal{A}$ such that $A \not \subset^* B$.
	
	\begin{prop}\label{prop:icpofideals}
		Let $\I$ be an ideal on $\omega$ that satisfies $\fin \subset \I$.
		Then we have $\icp(\I \setminus \fin, \subset^*) = \add^*(\I)$.
	\end{prop}
	\begin{proof}
		Let $\kappa = \add^*(\I)$ and let $\seq{A_\alpha : \alpha < \kappa}$ be a sequence of infinite $\I$-small sets such that
		\[
		\neg (\exists C \in \I)(\forall \alpha < \kappa) (A_\alpha \subset^* C).
		\]
		We construct a sequence $\seq{B_i : i < \kappa}$ of infinite $\I$-small sets such that
		\[
		B_i \cap B_{i+1} = \emptyset \text{ for every $i < \kappa$ and} \tag{$*$}
		\]
		\[
		\neg (\exists C \in \I)(\forall i < \kappa) (B_i \subset^* C). \tag{$**$}
		\]
		
		We claim that we can take such a sequence.
		We will construct not only $\seq{B_i : i < \kappa}$ but also $\seq{\alpha_i : i < \kappa}$.
		Assume we have constructed  $B_j$ and $\alpha_j$ for $j < i$.
		
		If $i = 0$, then put $\alpha_0 = 0$ and $B_0 = A_0$.
		If $i$ is limit, then put $\alpha_i = \sup_{j < i} \alpha_{j}$ and $B_i = A_{\alpha_i}$.
		
		Suppose $i$ is a successor ordinal.
		Find the minimum index $\beta$ such that $\neg (A_{\beta} \subset^* A_{\alpha_{i-1}})$ holds.
		And we put $\alpha_i = \gamma$ and $B_i = A_\gamma \setminus A_{\alpha_{i-1}}$.
		
		Then ($*$) is easily implied from the construction.
		We have to show ($**$).
		Suppose that $(\exists C \in \I)(\forall \alpha < \kappa) (B_\alpha \subset^* C)$ holds.
		Take $\alpha < \kappa$ arbitrarily.
		Take the minimum $i < \kappa$ such that $\alpha < \alpha_i$.
		This $i$ must be a successor ordinal. Write $i$ as $i = j + n$ where $j$ is a limit ordinal and $n \ge 1$ is a natural number.
		By the construction, we have $A_\alpha \subset^* A_{\alpha_{i-1}}$.
		
		Then we have
		\[
		A_\alpha \subset^* A_{\alpha_{i-1}} \subset B_{j} \cup B_{j+1} \cup \dots \cup B_{j+n} \subset^* C.
		\]
		Since $\alpha$ was chosen arbitrarily, this contradicts the choice of the sequence $\seq{A_\alpha : \alpha < \kappa}$.
		
		We claim that $\{ B_i : i < \kappa \}$ is an incomparable family.
		
		Take an element $C \in \I \setminus \fin$. Then by ($**$), we can find $i < \kappa$ such that $\neg (B_i \subset^* C)$.
		For this $i$, if we also have $\neg (C \subset^* B_i)$, then we are done.
		If $C \subset^* B_i$, then $C$ and $B_{i+1}$ are almost disjoint, in particular, they are incomparable.
	\end{proof}
	
	\section{Weakly $\omega_1$-dense ideals on $\omega_1$}\label{sec:idealonomega1}
	
	In Section \ref{sec:boolean}, we defined $\WD(B)$ for a Boolean algebra $B$ and showed $\WD(B) = \cp(B \setminus \{0, 1\})$ for an atomless Boolean algebra $B$.
	
	An ideal $\I$ on $\omega_1$ is said to be $\omega_1$ dense if the density of the Boolean algebra $\Pow(\omega_1)/\I$ is $\omega_1$.
	Let us define that an ideal $\I$ on $\omega_1$ is weakly $\omega_1$-dense when $\WD(\Pow(\omega_1)/\I) = \omega_1$ holds. 
	
	It is known that the consistency strength of the existence of an $\omega_1$-dense ideal on $\omega_1$ is $\omega$ many Woodin cardinals.
	So it is natural to ask what is the consistency strength of the existence of a weakly $\omega_1$-dense ideal on $\omega_1$.
	In this section, we answer this question.
	
	\begin{fact}[{{\cite[Theorem 3.1]{bhm1973}}}]\label{fact:bhm}
		Let $I$ be a normal ideal on $\omega_1$.
		Suppose that $\I \restrict A$ is not $\omega_1$ dense for every $A \in \I^+$.
		Then for every sequence $\seq{S_\alpha : \alpha < \omega_1}$ of $\I$-positive sets, there is a pairwise disjoint sequence $\seq{A_\alpha : \alpha < \omega_1}$ of $\I$-positive sets such that $A_\alpha \subset S_\alpha$ for every $\alpha < \omega_1$.
	\end{fact}
	
	\begin{thm}
		Let $\I$ be a normal, weakly $\omega_1$-dense ideal on $\omega_1$.
		Then $\I \restrict A$ is $\omega_1$-dense for some $A \in \I^+$.
	\end{thm}
	\begin{proof}
		Suppose that $\I \restrict A$ is not $\omega_1$ dense for every $A \in \I^+$.
		Let $\seq{S_\alpha : \alpha < \omega_1}$ be a sequence of $\I$-positive sets.
		Let us show that this family is not a weakly dense set.
		So we shall find $B \in \I^+$ such that $S_\alpha \not \subset_I B$ and $S_\alpha \not \subset_\I \omega_1 \setminus B$ for every $\alpha < \omega_1$.
		
		By Fact \ref{fact:bhm}, we can find a pairwise disjoint sequence $\seq{A_\alpha : \alpha < \omega_1}$ of $\I$-positive sets such that $A_\alpha \subset S_\alpha$ for every $\alpha < \omega_1$.
		Then we split each $A_\alpha$ into two positive sets $B_\alpha$, $C_\alpha$. This can be done using the fact that there is no $\sigma$-complete ultrafilter on $\omega_1$.
		Let $B$ be the union of $B_\alpha$'s.
		This $B$ is as required.
	\end{proof}
	
	\begin{cor}
		The consistency strength of the existence of a normal, weakly $\omega_1$-dense ideal on $\omega_1$ is also $\omega$ many Woodin cardinals. \qed
	\end{cor}
	
	\section{Questions}
	
	The following questions remain.
	
	\begin{question}
		\begin{enumerate}
			\item What are the values of $\cp((\nul \cap \Borel) \setminus \{\emptyset\})$ and $\cp((\meager \cap \Borel) \setminus \{\emptyset\})$?
			\item Can we prove $\cp(\I \setminus \fin, \subset^*) = \cof^*(\I)$ for every ideal on $\omega$? In particular, can we prove this inequality by Tukey reducibility?
			\item What are the values of $\mmac((\Borel(2^\omega)/\meager)^-)$ and $\mmac((\Borel(2^\omega)/\nul)^-)$?
			\item In Miller model, what are the values of $\cp(\beta \omega \setminus \omega, \leRK)$ and $\mmac(\beta \omega \setminus \omega, \leRK)$?
			\item Can we prove theorems in Section \ref{sec:nullideal} and \ref{sec:meagerideal} using Tukey reducibility?
		\end{enumerate}
	\end{question}
	
	\section{Acknowledgments}
	
	The author thanks Yusuke Hayashi for discussing this study with him.
	He also would like to thank Jörg Brendle, who suggested to him the idea of the proof of Theorem \ref{thm:cpnull}.
	In order to prove the results in Section \ref{sec:rudinkeisler}, The author was given helpful comments by Dilip Raghavan and Michael Hrušák.
	Theorem \ref{thm:mmacomegaomega} was obtained through private communication with Jorge Antonio Cruz Chapital.
	The result in Section \ref{sec:idealonomega1} is due to the private communication with Paul Larson.
	This work was supported by JSPS KAKENHI Grant Number JP22J20021.
	
	\nocite{*}
	\printbibliography[title={References}]
\end{document}